\documentclass[12pt,reqno]{amsart}
\usepackage{amsmath,amssymb,amstext,amsfonts,epsfig,graphicx}
\usepackage[left=3.5cm,top=3.5cm,right=3.5cm,bottom=3.5cm]{geometry}
\usepackage{float}
\usepackage{graphicx}
\usepackage{float}
\usepackage{morefloats}
\usepackage{float}
\usepackage{morefloats}
\usepackage{mathtools}
\usepackage[usenames,dvipsnames]{color}

\newtheorem{theorem}{Theorem}
\newtheorem{lemma}[theorem]{Lemma}

\newtheorem{corollary}[theorem]{Corollary}

\title{How to burn a graph}
\author{Anthony Bonato}
\address{Department of Mathematics\\
Ryerson University\\
Toronto, ON\\
Canada, M5B 2K3} \email{abonato@ryerson.ca}
\author{Jeannette Janssen}
\address{Department of Mathematics and Statistics\\
Dalhousie University\\
Halifax, NS\\
Canada, B3H 3J5}
\email{jeannette.janssen@dal.ca}
\author{Elham Roshanbin}
\address{Department of Mathematics and Statistics\\
Dalhousie University\\
Halifax, NS\\
Canada, B3H 3J5}
\email{e.roshanbin@Dal.Ca}
\thanks{Supported by grants from NSERC}

\begin{document}

\maketitle

\begin{abstract}
We introduce a new graph parameter called the burning number, inspired by contact processes on graphs such as graph bootstrap percolation, and graph searching paradigms such as Firefighter. The
burning number measures the speed of the spread of contagion in a graph; the lower the burning number, the faster the contagion spreads. We provide a number of properties of the burning number,
including characterizations and bounds. The burning number is computed for several graph classes, and is derived for the graphs generated by the Iterated Local Transitivity model for social
networks.
\end{abstract}

\section{Introduction}

The spread of social influence is an active topic in social network analysis; see, for example, \cite{shak,dr,kkt,ktt1,mr,rd}. A recent study on the spread of emotional contagion in
Facebook~\cite{kramer} has highlighted the fact that the underlying network is an essential factor; in particular, in-person interaction and nonverbal cues are not necessary for the spread of the
contagion. Hence, agents in the network spread the contagion to their friends or followers, and the contagion propagates over time. If the goal was to minimize the time it took for the contagion to
reach the entire network, then which agents would you target with the contagion, and in which order?

As a simplified, deterministic approach to these questions, we consider a new approach involving a graph process which we call \emph{burning}. Burning is inspired by graph theoretic processes like
Firefighting~\cite{winkler,bonato,fmac}, graph cleaning \cite{APW}, and graph bootstrap percolation~\cite{bbm}. There are discrete time-steps or rounds. Each node is either \emph{burned} or
\emph{unburned}; if a node is burned, then it remains in that state until the end of the process. In every round, we choose one additional unburned node to burn (if such a node is available). Once a
node is burned in round $t$, in round $t+1$, each of its unburned neighbours becomes burned. The process ends when all nodes are burned.  The \emph{burning number} of a graph $G$, written by $b(G)$,
is the minimum number of rounds needed for the process to end. For example, it is straightforward to see that $b(K_n) = 2$. However, even for a relatively simple graph such as the path $P_n$ on $n$
nodes, computing the burning number is more complex; in fact, $b(P_n) = \lceil n^{1/2}\rceil$ as stated below in Theorem~\ref{path}.

Burning may be viewed as a simplified model for the spread of social contagion in a social network such as Facebook or Twitter. The lower the value of $b(G)$, the easier it is to spread such
contagion in the graph $G.$ Suppose that in the process of burning a graph $G$, we eventually burned the whole graph $G$ in $k$ steps, and for each $i$, $1\leq i \leq k$, we denote the node that we
burn in the $i$-th step by $x_i$. We call such a node simply a \emph{source of fire}. The sequence $(x_1, x_2, \ldots, x_k)$ is called a \emph{burning sequence} for $G$. With this notation, the
burning number of $G$ is the length of a shortest burning sequence for $G$; such a burning sequence is referred to as \emph{optimal}. For example, for the path $P_4$ with node set $\{v_1, v_2, v_3,
v_4\}$, the sequence $(v_2, v_4)$ is an optimal burning sequence; see Figure~\ref{p4}. Note that for a graph $G$ with at least two nodes, we have that $b(G)\geq 2$.
\begin{figure}[hi]
\begin{center}
\epsfig{figure=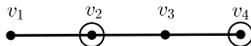,width=1.4in,height=.32in}
\caption{Burning the path $P_4$ (the open circles represent burned nodes).} \label{p4}
\end{center}
\end{figure}

The goal of the current paper is to introduce the burning number and explore its core properties. A characterization of burning number via a decomposition into trees is given in Theorem~\ref{decomp}.
As proven in \cite{bjr}, computing the burning number of a graph is \textbf{NP}-complete, even for trees. As such, we provide bounds on the burning number in terms of spanning trees, which are useful
in many cases when computing the burning number. See Lemma~\ref{rad} for bounds on the burning number. We compute the burning number on the Iterated Local Transitivity model for social networks
(introduced in \cite{anppc}); see Theorem~\ref{ilttt}. In the final section, we summarize our results and present open problems for future work.

\section{Properties of the burning number}

In this section, we collect a number of results on the burning number, ranging from characterizations, bounds, to computing the burning number on certain kinds of graphs. We first need some
terminology. If $G$ is a graph and $v$ is a node of $G$, then the \emph{eccentricity} of $v$ is defined as $\max\{d(v, u): u\in V(G)\}$. The \emph{radius} of $G$ is the minimum eccentricity over the
set of all nodes in $G$. The \emph{center} of $G$ consists of the nodes in $G$ with minimum eccentricity. Given a positive integer $k$, the \emph{$k$-th closed neighborhood} of $v$ is defined to be
the set $\{u\in V(G) : d(u, v)\leq k\}$ and is denoted by $N_k[v]$; we denote $N_1[v]$ simply by $N[v]$. We sometimes use the notation $N^G_k[v]$ to emphasize that we consider the $k$-th closed
neighbourhood of node $v$ in a specified graph $G$.

Suppose that $(x_1, x_2 \ldots, x_k)$ is a burning sequence for a given graph $G$. For $1\leq i\leq k$, the fire started at $x_i$ will burn only all the nodes within distance
$k-i$ from $x_i$ by the end of the $k$-th step. On the other hand, every node $v\in V(G)$ must be either a source of fire, or burned from at least one of the sources of fire by the end of the $k$-th
step. In other words, any node of $G$ must be an element of $N_{k-i}[x_i]$, for some $1\leq i\leq k$.
Moreover, for each pair $i$ and $j$, with $1\leq i < j\leq k$, we must have $d(x_i, x_j) \geq j-i$. Since, otherwise, if $d(x_i, x_j) = l < j-i$, then $x_j$ will be burned at stage $l + i$ ($<
j$), which is a contradiction.
Therefore, we can see that $(x_1, x_2, \ldots, x_k)$ forms a burning
sequence for $G$ if and only if, for each pair $i$ and $j$, with $1\leq i < j\leq k$, $d(x_i, x_j) \geq j-i$, and the following set equation holds:
\begin{equation}
N_{k-1}[x_1] \cup N_{k-2}[x_2]\cup \ldots \cup N_0[x_k] = V(G).\label{eqq}
\end{equation}
A \emph{covering} of $G$ is a set of subsets of the nodes of $G$ whose union is $V(G).$  The above observation, shows that the burning problem is basically a covering problem using closed
neighbourhoods with a restriction on their radius. Hence, it seems that by finding a covering for a graph $G$ using a limited number of connected subgraphs with restricted
radius, we may find a bound on the burning number of $G$, as the following theorem shows.

\begin{theorem}\label{corr1}
If in a graph $G$ there exists a collection of connected subgraphs $\{C_1, C_2,\ldots, C_t\}$, each of radius at most $k$, which cover all the nodes of $G$, then $b(G)\leq t + k$.
\end{theorem}

\begin{proof}
It suffices to show that for any minimal covering of $G$ (that is a covering which does not contain any proper subcover for $G$) such as $\{C_1, C_2,\ldots, C_t\}$ by nonempty connected subgraphs
$C_i$'s of radius at most $k$, that $b(G)\leq t + k$. We define a burning sequence $(x_1, x_2, \ldots, x_{t' + k'})$, with $t'\leq t$ and $k'\leq k$, for $G$ as follows. Let $x_1$ be a center of the
induced subgraph $G[C_1]$. Then for $i\geq 2$, we let $x_i$ be a central node in $C_j$, with $j\geq i$, if none of the central nodes of $C_j$ are burned before the $i$-th step, where $j$ is the
smallest index that satisfies this condition. We continue to choose $x_i$'s by the above rule until at some step $t'\leq t$, by burning $x_1, x_2, \ldots, x_{t'}$, each $C_i$, $1\leq i\leq t$
contains a burned center.

Now, for $j\geq 1$, we choose $x_{t' + j}$ to be a node in $G$ that is not burned before the $(t' + j)$-th step. Since the radius of each $C_i$ is at most $k$, after $k'\leq k$ steps every node in
$G$ must be burned. Thus, $b(G)\leq t' + k'\leq t + k$.
\end{proof}

We present another bound for the burning number of a graph using coverings. The proof is analogous to the one of Theorem~\ref{corr1}, and so is omitted.
\begin{theorem}\label{cov2}
If $\{C_1, C_2,\ldots, C_t\}$ is a covering of the nodes of a graph $G$, where each $C_i$ is a connected subgraph of radius at most $k-i$, and $t\leq k$, then $b(G)\leq k$.
\end{theorem}

We have the following immediate corollary.

\begin{corollary}\label{cov3}
If $(x_1, x_2, \ldots, x_k)$ is a sequence of nodes in a graph $G$, such that $N_{k-1}[x_1] \cup N_{k-2}[x_2]\cup \ldots \cup N_0[x_k] = V(G)$, then $b(G) \leq k$.
\end{corollary}

\begin{proof}
Set $C_i = N_{k-i}[x_i]$, for $1\leq i\leq k$, and apply Theorem \ref{cov2}.
\end{proof}

Another immediate corollary is the following.

\begin{corollary}\label{cone}
Suppose that $(x_1, x_2, \ldots, x_k)$ is a burning sequence for a graph $G$. If for some node $x\in V(G) \setminus \{x_1, \ldots, x_k\}$ and  $1\leq j \leq k-1$, we have that $N[x] \subseteq
N[x_j]$, and for every $i\neq j$, $d(x,x_i)\geq |i-j|$, then $(x_1, \ldots, x_{j-1}, x, x_{j+1}, \ldots, x_k)$ is also a burning sequence for $G$.
\end{corollary}
We consider the burning problem for \emph{connected graphs}. Note that, as is the case for many graph parameters, the burning number of a disconnected graph $G$ with components $G_1, G_2, \ldots,
G_t$, where $t\geq 2$, does not necessarily satisfy the equality $b(G) = b(G_1) + b(G_2) + \cdots + b(G_t)$. For example, let $G$ be the disjoint union of $t$ paths of order $2$, which has burning
number $t+1.$

The following theorem provides an alternative characterization of the burning number. The \emph{depth}
of a node in a rooted tree is the number of edges in a shortest path from the node to the tree's root. The \emph{height} of a rooted tree $T$ is the greatest depth in $T$. A \emph{rooted tree
partition} of $G$ is a collection of rooted trees which are subgraphs of $G$, with the property that the node sets of the trees partition $V(G)$.

\begin{theorem}\label{decomp}
Burning a graph $G$ in $k$ steps is equivalent to finding a rooted tree partition into $k$ trees $T_1, T_2,\ldots, T_k$, with heights at most $(k-1), (k-2), \ldots, 0$, respectively such that for
every $1\le i , j \le k$ the distance between the roots of $T_i$ and $T_j$ is at least $|i-j|$.
\end{theorem}
\begin{proof}
Assume that $(x_1, x_2,\ldots, x_k)$ is a burning sequence for $G$. For all $1\leq i\leq k$, after $x_i$ is burned, in each round $t> i$ those unburned nodes of $G$ in the $(t-i)$-neighborhood of
$x_i$ will burn. Hence, any node $v$ is burned by receiving fire via a shortest path of burned nodes from a fire source like $x_i$ (this path can be of length zero in the case that $v = x_i$). Hence,
we may define a surjective function $f: V(G)\rightarrow \{x_1, x_2,\ldots, x_k\}$, with $f(v)= x_i$ if $v$ receives fire from $x_i$, where $i$ is chosen with the smallest index. Now $\{ f^{-1} (x_1),
f^{-1} (x_2), \ldots, f^{-1} (x_k)\}$ forms a partition of $V(G)$ such that $G[f^{-1} (x_i)]$ (that is, the subgraph induced by $f^{-1} (x_i)$) forms a connected subgraph of $G$. Since every node $v$ in
$f^{-1} (x_i)$ receives the fire spread from $x_i$ through a shortest path between $x_i$ and $v$, by deleting extra edges in $G[f^{-1} (x_i)]$ we can make a rooted subtree of $G$, called $T_i$ with
root $x_i$. Since every node is burned after $k$ steps, the distance between each node on $T_i$ and $x_i$ is at most $k-i$. Therefore, the height of $T_i$ is at most $k-i$.
\begin{figure}[ht]
\begin{center}
\epsfig{figure=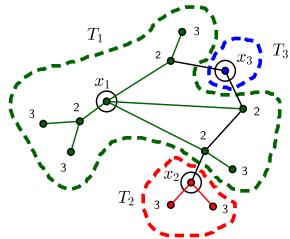,width=1.5in,height=1.3in}
\end{center}
\caption{A rooted tree partition.}\label{egburn}
\end{figure}

Conversely, suppose that we have a decomposition of the nodes of $G$ into $k$ rooted subtrees $T_1, T_2,\ldots, T_k$, such that for each $1\leq i \leq k$, $T_i$ is of height at most $k- i$. Assume
that $x_1, x_2,\ldots, x_k$ are the roots of $T_1, T_2,\ldots, T_k$, respectively, and for each pair $i$ and $j$, with $1\leq i < j\leq k$, $d(x_i, x_j) \geq j-i$. Then $(x_1, x_2,\ldots, x_k)$ is a
burning sequence for $G$, since the distance between any node in $T_i$ and $x_i$ is at most $k-i$. Thus, after $k$ steps the graph $G$ will be burned.
\end{proof}

Figure~\ref{egburn} illustrates Theorem~\ref{decomp}. The burning sequence is $(x_1, x_2, x_3)$. We have shown the decomposition of $G$ into subgraphs $T_1,$ $T_2,$ and $T_3$ based on this burning
sequence by drawing dashed curves around the corresponding subgraphs. Each node has been indexed by a number corresponding to the step that it is burned.

The following corollary is useful for determining the burning number of a graph, as it reduces the problem of burning a graph to burning its spanning trees. Note that for a spanning subgraph $H$ of
$G$, it is evident that $b(G)\leq b(H)$. This follows since, by equation (\ref{eqq}), every optimal burning sequence for $H$ induces a node covering for $V(G) = V(H)$, and therefore, by Corollary
\ref{cov3}, induces a burning sequence of at most the same length for $G$.

\begin{corollary}\label{spann}
For a graph $G$ we have that
$$b(G) = \min\{b(T): T \text{ is a spanning subtree of } G \}.$$
\end{corollary}

\begin{proof}
By Theorem~\ref{decomp}, we assume that $T_1, T_2,\ldots, T_k$ is a rooted tree partition of $G$, where $k=b(G)$, derived from an optimal burning sequence for $G$. If we take $T$ to be a spanning
subtree of $G$ obtained by adding edges between the $T_i$'s which do not induce a cycle in $G$, then $b(T)\leq k = b(G) \leq b(T)$, where the second inequality holds since $T$ is a
spanning subgraph of $G$.
\end{proof}

\medskip

A subgraph $H$ of a graph $G$ is called an \emph{isometric subgraph} if for every pair of nodes $u$ and $v$ in $H$, we have that $d_H(u,v) = d_G(u,v)$. For example, a subtree of a tree is an
isometric subgraph. As another example, if $G$ is a connected graph and $P$ is a shortest path connecting two nodes of $G$, then $P$ is an isometric subgraph of $G$. Let $W_5$ be the wheel graph
formed by adding a universal node to a $5$-cycle. Then, the $5$-cycle $C_5$ is an isometric subgraph of $W_5$, while $b(C_5) = 3 > 2 = b(W_n)$. Thus, we conclude that the burning number is not
monotonic even on the isometric subgraphs of a graph. However, the following theorem shows that the burning number is monotonic on the isometric subgraphs in certain cases.

\begin{theorem}\label{isometric1}
Let $H$ be an isometric subgraph of a graph $G$ such that, for any node $x\in V(G) \setminus V(H)$, and any positive integer $r$, there exists a node $f_r(x) \in V(H)$ for which $N_r[x]\cap V(H)
\subseteq N^H_r[f_r(x)] $. Then we have that $b(H) \leq b(G)$.
\end{theorem}

\begin{proof}
It suffices to show that for any optimal burning sequence of $G$ such as $(x_1, x_2, \ldots, x_k)$ we can assign a burning sequence of length at most $k$ to $H$. Without loss of generality, we may
assume that $|V(H)|> k$ (otherwise, $H$ can be burned in at most $|V(H)|\leq k$ steps). By equation (\ref{eqq}), $N_{k-1}[x_1] \cup N_{k-2}[x_2]\cup \ldots \cup N_0[x_k] = V(G)$.

We define the function $f: \{x_1, x_2, \ldots, x_k\} \rightarrow V(H)$ as follows. For $1\leq i\leq k$, if $x_i\in V(H)$, then we define $f(x_i) = x_i$; otherwise, by assumption, there is a node
$f_{k-i}(x_i)\in V(H)$ for which $N_{k-i}[x_i] \cap V(H) \subseteq N^H_{k-i}[f_{k-i}(x_i)]$. In this case, we define $f(x_i) = f_{k-i}(x_i)$. Since $H$ is an isometric subgraph of $G$, then for each
node $x_i$ with $f(x_i) = x_i$, and for every node $v \in N_{k-i}[x_i]\cap V(H)$, we have that $d_H(x_i, v) =  d_G(x_i, v) \leq k-i $. Thus, if $f(x_i) = x_i$, then $N_{k-i}[x_i] \cap V(H) =
N^H_{k-i}[x_i] = N^H_{k-i}[f(x_i)]$. Hence, we derive that
\begin{align*}
V(H) & = V(G) \cap V(H)\\
& = \left( N_{k-1}[x_1] \cup \ldots \cup N_0[x_k] \right) \cap V(H)\\
& = \left( N_{k-1}[x_1] \cap V(H)\right) \cup \ldots \cup \left( N_0[x_k] \cap V(H)\right) \\
& \subseteq  N^H_{k-1}[f(x_1)] \cup \ldots \cup N^H_0[f(x_k)].
\end{align*}
Therefore, $\{N_{k-i}[f(x_i)]\}_{i=1}^k$ forms a covering for the node set of $H$, with $k$ closed neighbourhoods. Thus, by
Corollary \ref{cov3}, we conclude that $b(H) \leq b(G)$.
\end{proof}

The following theorem shows that the isometric subtrees of a graph satisfy the conditions in Theorem \ref{isometric1}.

\begin{theorem}\label{isometric2}
For any isometric subtree $H$ of a graph $G$, we have that $b(H) \leq b(G)$.
\end{theorem}

\begin{proof}
By Theorem \ref{isometric1}, it suffices to show that for any node $x\in V(G) \setminus V(H)$, and any positive integer $r$, there exists a node $f_r(x) \in V(H)$ for which $N_r[x]\cap V(H) \subseteq
N^H_r[f_r(x)]$.

Assume that $X_r = N_r[x] \cap V(H)$. If $X_r$ is empty, then we can choose $f_r(x)$ to be any node in $H$. If $X_r = \{v\}$, then we take $f_r(x) = v$, and clearly $N_r[x]\cap V(H) = \{v\} \subseteq
N^H_r[v]$. Hence, we assume that $|X_r|\geq 2$. Since $H$ is a tree, then there is a unique path (consisting of the nodes in $H$ only) between every pair of distinct nodes in $X_r\subseteq V(H)$. Let
$y_r$ and $z_r$ be two nodes in $X_r$ with the maximum distance over all possible pairs of nodes in $X_r$, and let $w_r$ be a node in $H$ that is of almost equal distance with respect to $y_r$ and
$z_r$. That is, $d(w_r,y_r) = d(w_r, z_r)$, if $d(y_r,z_r)$ is even, and $d(w_r, z_r) = d(w_r,y_r) + 1$ (without of loss of generality) in the case that $d(y_r,z_r)$ is odd.
We claim that for each $v\in X_r$, $d(v, w_r) \leq r$.

Since, by assumption, $H$ is an isometric subtree of $G$, then the length of the path between $y_r$ and $z_r$ in $H$ is equal to $d(y_r, z_r)$ in $G$. Thus, we have that $d(y_r, z_r) = d(w_r,
y_r)+d(w_r, z_r) \leq d(x , y_r) + d(x, z_r)$. On the other hand, we have that $d(v, w_r) \leq d(z_r, w_r)$. To show this, we have to consider two possibilities; either $v$ is on the path in $H$ that
connects $z_r$ to $w_r$, or it is not. If the former holds, then obviously, $d(v, w_r) \leq d(z_r, w_r)$. If the latter holds, then suppose $u$ is the first node that appears in both paths that
connect $v$ and $z_r$ to $w_r$. If $u= w_r$, then we have that
\begin{align*}
d(v, w_r) + d(w_r , z_r) = d(v, z_r) \leq d(y_r, z_r) = d(y_r , w_r ) + d(w_r , z_r).
\end{align*}
It implies that $d(v, w_r) \leq d(y_r, w_r) \leq d(z_r, w_r)$.
If $u \neq w_r$, then we have that
\begin{align*}
d(v, w_r) + d(w_r , y_r) = d(v, y_r) \leq d(z_r, y_r) = d(z_r , w_r ) + d(w_r , y_r).
\end{align*}
Hence, we again conclude that $d(v, w_r) \leq d(z_r, w_r)$. Consequently, we have that
\begin{align*}
d(v , w_r) & \leq d(z_r , w_r) \\
& \leq \frac{d(y_r , z_r) + 1}{2} \\
& \leq \frac{d(y_r , x) + d(z_r, x) +1}{2} \\
& \leq \frac{r + r + 1}{2} \\
&= r + \frac{1}{2}.
\end{align*}
Since $d(v , w_r)$ is an integer, it implies that $d(v , w_r) \leq r$. Therefore, if we define $f_r(x) = w_r$, then $X_r \subseteq N^H_r[f_r(x)]$. Thus, the proof follows.
\end{proof}

However, the above inequality may fail for non-isometric subtrees. For example, let $H$ be a path of order $5$, and form $G$ by adding a universal node (that is, one joined to all others) to $H.$
Then $b(H)=3$, but $b(G)=2$.

The following corollary is an immediate consequence of Theorem~\ref{isometric2}.

\begin{corollary}\label{subtree}
If $T$ is a tree and $H$ is a subtree of $T$, then we have that $b(H)\leq b(T).$
\end{corollary}

The burning number of paths is derived in the following result.

\begin{theorem}\label{path}
For a path $P_n$ on $n$ nodes, we have that $b(P_n) = \lceil n^{1/2}\rceil$.
\end{theorem}
\begin{proof}
Suppose that $(x_1, x_2, \ldots, x_k)$ is an optimal burning sequence for $P_n$. By equation~(\ref{eqq}), and the fact that for a node $v$ in a path, $|N_i[v]|\leq 2i + 1$ we derive that
\begin{align*}
& \left( 2(k-1) + 2(k-2) + \ldots + 2(1)\right)  + k  \\
& = 2 \left( \frac{k(k-1)}{2} \right) + k \\
& = k^2 \geq n.
\end{align*}
Since $k$ is the minimum number satisfying this inequality, we conclude that $b(P_n) \geq  \lceil n^{1/2}\rceil$.

Now, assume that $k=  \lceil n^{1/2}\rceil$, and let $P_n: v_1, v_2, \ldots, v_n$. Then for $0\leq i \leq k-2$, we choose $x_{k-i} = v_{n- i^2 -i}$. Also, if $n\geq (k-1)^2 + k$, we take $x_1 = v_{n-
(k-1)^2 -(k -1)}$; otherwise we take $x_1 = v_1$. Therefore, we can burn $P_n$ in exactly $k$ steps by the burning sequence $(x_1, x_2, \ldots, x_k)$. Hence, $b(P_n)\leq k$. Thus, $b(P_n) = \lceil
n^{1/2}\rceil$.
\end{proof}

We have the following immediate corollaries.

\begin{corollary}
\begin{enumerate}
\item For a cycle $C_n$, we have that $b(C_n) = \lceil n^{1/2}\rceil$.
\item For a graph $G$ of order $n$ with a Hamiltonian (that is, a spanning) path, we have that $b(G) \leq \lceil n^{1/2}\rceil$.
\end{enumerate}
\end{corollary}

The following theorem gives sharp bounds on the burning number. For $s\ge 3$, let $K_{1, s}$ denotes a \emph{star}; that is, a complete bipartite graph with parts of order $1$ and $s$. We call a
graph obtained by a sequence of subdivisions starting from $K_{1,s}$ a \emph{spider graph}. In a spider graph $G$, any path which connects a leaf to the node with maximum degree is called an
\emph{arm} of $G$. If all the arms of a spider graph with maximum degree $s$ are of the same length $r$, we denote such a spider graph by $SP(s,r)$.

\begin{lemma}\label{rad}
For any graph $G$ with radius $r$ and diameter $d$, we have that
$$\lceil (d+1)^{1/2}\rceil\leq b(G)\leq r+1.$$
\end{lemma}

\begin{proof}
Assume that $c$ is a central node of $G$ with eccentricity $r$. Since every node in $G$ is within distance $r$ from $c$, the fire will spread to all nodes after $r+1$ steps. Hence, $r+1$ is an upper
bound for $b(G)$.

Now, let $P$ be a path connecting two nodes $u$ and $v$ in $G$ with $d(u,v)=d$. Since $P$ is an isometric subtree of $G$, and $|P|= d+1$, by Theorem \ref{isometric2} and Theorem \ref{path}, we
conclude that $b(G)\geq b(P) = \lceil (d+1)^{1/2}\rceil$.
\end{proof}

The lower bound is achieved by paths, and the right side bound is achieved by spider graphs $SP(s,r)$, where $s\geq r$ (as proven in \cite{bjr}).



\medskip

We finish this section by providing some bounds on the burning number in terms of certain domination numbers.  A \emph{$k$-distance dominating set} like $D_k$ for $G$ is a subset of nodes such that
for every node $u\in V(G) \setminus D_k$, there exists a node $v\in D_k$, with $d(u,v)\leq k$. The number of the nodes in a minimum $k$-distance dominating set of $G$ is denoted by $\gamma_k(G)$ and
we call it the \emph{$k$-distance domination number} of $G$. We have the following results on connections between burning and distance domination.

\begin{theorem}\label{gamb}
If $G$ with burning number $k$, then we have that $k\geq \gamma_{k-1}(G)$.
\end{theorem}
\begin{proof}
Assume that $b(G)=k$, for some positive integer $k$, and $(x_1, x_2, \ldots, x_k)$ is an optimal burning sequence for $G$. Then by (\ref{eqq}), we know that every node $v$ in $G$ must be
within the distance $k-i \leq k-1$ from one of the $x_i$'s. Hence, $D= \{x_1, x_2, \ldots, x_k\}$ forms a $(k-1)$-distance dominating set for $G$.
\end{proof}

We have the following lemma.
\begin{lemma}\label{mingammak}
For any graph $G$, if $m = \min_{k\geq 1}\{ \gamma_k (G) + k\}$, then $\frac{m+1}{2}\leq b(G) \leq m$.
\end{lemma}
\begin{proof}
Assume that $m = \min_{k\geq 1}\{ \gamma_k (G) + k\}$, and $b(G)=k_0$. Then by Theorem \ref{gamb}, $b(G) = k_0 \geq \gamma_{k_0 -1}$. Hence, $$k_0 + (k_0 -1) \geq \gamma_{k_0 -1} + k_0 -1 \geq
\min_{k\geq 1}\{ \gamma_k (G) + k\}= m .$$ Therefore, $k_0 \geq \frac{m+1}{2} $.

On the other hand, assume that $D_k =\{x_1, x_2, \ldots, x_{\gamma_k}\}$ is a minimum $k$-distance dominating set for $G$. Then $\{S_1, S_2,\ldots, S_{\gamma_k}\}$, with $S_i = \{v\in V(G): d(v,
x_i)\leq k\}$, where $1\leq i\leq \gamma_k$,  is a covering for the nodes of $G$ which consists of $\gamma_k(G)$ subsets each of radius at most $k$. Thus, by Corollary \ref{corr1}, we have that
$b(G)\leq \gamma_k + k$. The result follows since this is true for any $k\geq 1$. Therefore, we have that
$$\frac{m+1}{2}\leq b(G)\leq m,$$
and the proof follows.
\end{proof}

We have the following fact about the $k$-distance domination number of graphs.

\begin{theorem}\cite{mm}\label{upgammak}
If $G$ is a connected graph of order $n$, with $n\geq k+1$, then we have that
$$\gamma_k(G) \leq \frac{n}{k+1}.$$
\end{theorem}
We now provide the following general bound for the burning number of graphs.

\begin{corollary}
If $G$ is a connected graph of order $n$, with $n\geq k+1$, then we have that
$$b(G) \leq 2\lceil{n}^{1/2}\rceil - 1.$$
\end{corollary}

\begin{proof}
By Lemma~\ref{mingammak} and Theorem~\ref{upgammak}, we derive that for any positive integer $k\leq n-1$ $$b(G)\leq \min_{k\geq 1}\left\{ \frac{n}{k+1} + k\right\}.$$
Now, the function $\frac{n}{k+1} + k$ is
minimized for $k=\lceil{n}^{1/2}\rceil - 1$, and we note that $k\leq n -1$. Therefore, we have that
\begin{align*}
b(G) & \leq \min_{k\geq 1}\left\{ \frac{n}{k+1} + k\right\}\\
& \leq \frac{n}{(\lceil{n}^{1/2}\rceil - 1)+1} + \lceil{n}^{1/2}\rceil - 1 \\
& \leq 2\lceil{n}^{1/2}\rceil - 1,
\end{align*}
and the proof follows.
\end{proof}
We conjecture that for any connected graph $G$ of order $n$, $b(G)\leq \lceil {n}^{1/2}\rceil$.

\section{Nordhaus-Gaddum Type Results}

Nordhaus and Gaddum~\cite{ng} gave bounds on the sum and product of the chromatic number of a graph and its complement, in terms of the order of the graph. Analogous relations have been discovered
for many other graph parameters; see \cite{ao} for a survey. In this section, we present Nordhaus-Gaddum type results for the burning number.

We need first the following simple observation. Let $G$ be a graph of order $n\geq 2$ with maximum degree $\Delta$. If $G$ does not have a universal node, then we have that $b(G)\leq n - \Delta$;
otherwise, $b(G) = 2$. It follows since we can take a node such as $v$ of degree $\Delta$, and then, by burning $v$ and $V(G) \setminus N[v]$, respectively, we burn all nodes of $G$ in at most $1 +
|V(G) \setminus N[v]| = n - \Delta$ steps. If $G$ contains a universal node $v$, then by burning $v$ and one of its neighbours respectively, we can burn $G$ in two steps (note that since $G$ has an
edge, we need at least two steps for burning $G$). Thus, $b(G) =2$. Also, we need the following theorem from \cite{bjr}.
\begin{theorem}[\cite{bjr}]\label{b222}
A graph $G$ satisfies $b(G)=2$ if and only if $G$ has order at least $2$, and has maximum degree $n-1$ or $n-2.$
\end{theorem}

We first present some bounds on the sum of the burning
numbers of a graph and its complement.
\begin{theorem}\label{sumcomp}
If $G$ is a graph of order $n\geq 2$, then $$4 \le b(G) + b(\overline{G})\leq n + 2.$$
\end{theorem}

\begin{proof}
Suppose that $b(G) = k$, and $(x_1, x_2, \ldots, x_k)$ is a burning sequence for $G$. Clearly, $x_k$ cannot be adjacent to $x_i$, for $1\leq i\leq k-2$. Therefore, $\Delta(\overline{G})\geq
d_{\overline{G}}(x_k)\geq k-2$. If $G$ does not have an isolated node, then $\overline{G}$ does not have a universal node. Thus, by the above observation, $b(\overline{G}) \leq n -
\Delta(\overline{G}) \leq n -(k-2)$, and consequently, we have that $b(\overline{G}) + b(G) \leq n + 2$. Moreover, in such a case, both $G$ and $\overline{G}$ must have at least one edge. Thus,
$b(\overline{G}) + b(G) \geq 2+2 = 4$.

If $G$ has an isolated node $v$, then $\overline{G}$ must have a universal node. Therefore, $b(\overline{G})=2$, and clearly, $b(G) \leq n$. Thus, $b(\overline{G}) + b(G) \leq n + 2$. On the other
hand, since by assumption, $G$ has at least two nodes and one of them is isolated, then $b(G) \geq 2$. Hence, $b(\overline{G}) + b(G) \geq 2+2 = 4$.
\end{proof}

The upper bound in Theorem \ref{sumcomp} is the best possible, since for the complete graph $K_n$ we have that $ b(\overline{K_n}) + b(K_n) = n +2$. However, there are cases where the upper bound is
strict; for example, $b(C_n) + b(\overline{C_n}) = \lceil {n}^{1/2}\rceil + 3 < n + 2$. Also, the lower bound in Theorem \ref{sumcomp} is achieved for the complete graph $K_2$, and star graphs
$K(1,s)$.

We cite here two useful Nordhaus-Gaddum type results for distance domination.

\begin{theorem}\cite{ao}\label{gadamdistance1}
For any graph $G$ of order $n\geq k+1$ with $k\geq 2$, we have that $$\gamma_k(G) + \gamma_k(\overline{G})\leq n+1$$ and $$\gamma_k(G)\gamma_k(\overline{G})\leq n. $$
\end{theorem}

\begin{theorem}\cite{ao}\label{gadamdistance2}
If $G$ and $\overline{G}$ are both connected with $n\geq k+1$ nodes for integer $k\geq 2$, then $\gamma_k(G) + \gamma_k(\overline{G})\leq \frac{n}{k+1} + 1$ and
$\gamma_k(G)\gamma_k(\overline{G})\leq \frac{n}{k+1}$.
\end{theorem}

We now have the following result for the product of the burning numbers of a graph and its complement.

\begin{theorem}\label{bb}
For any graph $G$ of order $n\ge 6$, we have $b(G)b(\overline{G})\leq 2n$, and the equality is achieved by complete graphs.
\end{theorem}
\begin{proof}
By Corollary \ref{mingammak}, we have that $b(G)\leq \gamma_k(G) + k$, for $k\geq 2$ with $n\geq k+1$. Thus, we have that
\begin{align*}
b(G)b(\overline{G}) &\leq (\gamma_k(G) + k)(\gamma_k(\overline{G}) + k)\\
&\leq \gamma_k(G)\gamma_k(\overline{G}) + k\gamma_k(G) + k\gamma_k(\overline{G}) + k^2.
\end{align*}
Now, if $G$ and $\overline{G}$ are both connected, then using Theorem \ref{gadamdistance2}, we have that
\begin{align*}
b(G)b(\overline{G})& \leq \gamma_k(G)\gamma_k(\overline{G}) + k\gamma_k(G) + k\gamma_k(\overline{G}) + k^2\\
&\leq \frac{n}{k+1} + k \left(\frac{n}{k+1} + 1\right) + k^2 \\
& = n + k + k^2.
\end{align*}
By taking $k=2$, the above inequality implies that $b(G)b(\overline{G}) \leq n + 6$.

If $G$ is connected while $\overline{G}$ is disconnected, then, either $\overline{G}$ has a component with at most two nodes, or every component of
$\overline{G}$ has at least three nodes. If $\overline{G}$ has a component with at most two nodes, then $G$ must have either a universal node, or it must contain two nodes
such as $u$ and $v$ (corresponding to the component of $\overline{G}$ with exactly two nodes), such that $G = N[u] \cup \{v\}$. Thus, by Theorem \ref{b222}, $b(G)=2$, and obviously
$b(\overline{G})\leq n$. Hence, in this case, $b(G)b(\overline{G})\leq 2n$.

Now, suppose that $G_1, G_2, \ldots, G_t$ are the components of $\overline{G}$ with $n_1, n_2, \ldots, n_t$ nodes, respectively, where each $n_i\geq 3$. By Theorem \ref{upgammak}, and taking $k=2$, we know that
\begin{align*}
\gamma_2(\overline{G}) & = \gamma_2(G_1) +  \gamma_2(G_2) + \cdots +  \gamma_2(G_t)\\
& \leq \frac{n_1}{3} + \frac{n_2}{3} + \cdots + \frac{n_t}{3}\\
& = \frac{n}{3}.
\end{align*}
Also, note that $b(G)\leq 3$, since we can easily see that in such a case the radius of $G$ is at most $2$. Therefore, by Theorem \ref{gamb} and Theorem \ref{upgammak}, for $k=2$, we have that
\begin{align*}
b(G)b(\overline{G})& \leq 3(\gamma_2(\overline{G}) + 2) \\
&\leq 3 \left(\frac{n}{3} +  2\right) \\
& = n + 6,
\end{align*}
and the proof follows.
\end{proof}

\begin{corollary}
If graphs $G$ and $\overline{G}$ are connected graphs of order $n\ge 6$, then
$b(G) + b(\overline{G})\leq 3\lceil n^{1/2}\rceil -1 $, and
$b(G)b(\overline{G})\leq n+6$.
\end{corollary}

\begin{proof}
First, by Lemma \ref{gamb}, we have that
\begin{align*}
b(G) + b(\overline{G})& \leq (\gamma_k(G) +k) + (\gamma_k(\overline{G}) +k)\\
& = \gamma_k(G) + \gamma_k(\overline{G}) + 2k.
\end{align*}
By applying Theorem \ref{gadamdistance2} with $k= \lceil n^{1/2}\rceil - 1$, we conclude that
\begin{align*}
b(G) + b(\overline{G})& \leq \frac{n}{k+1} + 1 + 2k \\
& \leq 3\lceil n^{1/2} \rceil  -1.
\end{align*}
Finally, $b(G)b(\overline{G})\leq n+6$ follows from the proof of Theorem \ref{bb}.
\end{proof}

We conjecture that if $G$ and $\overline{G}$ are both connected graphs of order $n$, then we have that $b(G)b(\overline{G})\leq n + 4$.

\section{Burning in the ILT Model}

The \emph{Iterated Local Transitivity} (ILT) model was introduced in \cite{anppc}, and simulates
on-line social networks (or OSNs). The central idea behind the ILT model is
what sociologists call \emph{transitivity}: if $u$ is a friend of $v$, and $v
$ is a friend of $w,$ then $u$ is a friend of $w$. In its simplest form, transitivity gives rise to the
notion of \emph{cloning}, where $u$ is joined to all of the neighbours of $v$. In the ILT model, given some initial graph as a starting point, nodes are
repeatedly added over time which clone \emph{each} node, so that the new
nodes form an independent set. The only parameter of the model is the
initial graph $G_{0},$ which is any fixed finite connected graph. Assume
that for a fixed $t\geq 0,$ the graph $G_{t}$ has been constructed. To form $%
G_{t+1},$ for each node $x\in V(G_{t}),$ add its \emph{clone} $x^{\prime },$
such that $x^{\prime }$ is joined to $x$ and all of its neighbours at time $%
t.$ Note that the set of new nodes at time $t+1$ form an independent set of
cardinality $|V(G_{t})|.$

The ILT model shares many properties with OSNs such as low average distance, high clustering coefficient densification, and bad spectral expansion; see
\cite{anppc}. The ILT model has also been studied from the viewpoint of competitive diffusion which is one model of the spread of influence; see \cite{so}.

We have the following theorem about the burning number of graphs obtained based on ILT model.  Even though the graphs generated by the ILT model grow exponentially in order with $t$, we see that the
burning number of such networks remains constant.
\begin{theorem}\label{ilttt}
Let $G_t$ be the graph generated at time $t\geq 1$ based on the ILT model with initial graph $G_0$. If $G_0$ has an optimal burning sequence $(x_1, x_2, \ldots, x_k)$ in which $x_k$ has a neighbor that is
burned in the $(k-1)$-th step, then $b(G_t) = b(G_0)$. Otherwise, $b(G_t) = b(G_0) + 1$.
\end{theorem}
\begin{proof}
First, assume that $(x_1, x_2,\ldots, x_k)$ is an
optimal burning sequence for $G_0$. Since every node $x'\in V(G_t) \setminus  V(G_0)$, with $t\geq 1$, is adjacent to a node in $G_0$, we have that $(x_1, x_2,\ldots, x_k)$ is also a burning sequence for the subgraph of $G_t$
induced by $V(G_t) \setminus (N^{G_t}[x_k] \setminus N^{G_0}[x_k])$. Thus, $b(G_t)\leq b(G_0) + 1$. With a similar argument, we conclude that $b(G_t)\leq b(G_{t-1}) +1 $.

On the other hand, we can easily see that $G_{t-1}$ is an isometric subgraph of $G_t$, for any $t\geq 1$. Also, for any $x\in V(G_{t-1})$ and its clone $x' \in V(G_t)$, we have that $N^{G_t}[x] =
N^{G_t}[x']$. Thus, for any $r\geq 1$, $N_r[x'] \cap V(G_{t-1}) = N^{G_{t-1}}_r[x]$. Therefore, by Theorem \ref{isometric1}, $b(G_t)\geq b(G_{t-1})$. Hence, by induction we conclude that $b(G_t) \geq
b(G_0)$, for any $t\geq 1$, and therefore, we have that either $b(G_t) = b(G_0)$, or $b(G_t) = b(G_0) + 1$. We now characterize where $b(G_t)$ equals $b(G_0)$ or $b(G_t) = b(G_0) + 1$ as follows.

Let $(x_1, x_2,\ldots, x_k)$ be an optimal  burning sequence for $G_t$.
By the following algorithm, we find a burning sequence $(y_1, y_2,\ldots, y_k)$ for $G_{t}$ where at least all the first $k-1$ fire sources are in $G_{t-1}$.

\medskip

{\bf Step 1.} If $x_1 \in V(G_{t-1})$, then we take $y_1 = x_1$, and $x_i = x_i$ for $2\leq i\leq k$. Go to Step 2.

If for some node $x'_1 \in V(G_{t-1})$, the node $x_1$ is the clone of $x'_1$, then by Corollary \ref{cone}, we set $(y_1 = x'_1, x_2 = x_2, \ldots, x_k = x_k)$ as a new burning sequence for $G_t$. Go to Step 2.

\smallskip

{\bf Step 2.} For $2\leq i \leq k-1$, perform the following steps.

\smallskip

{\bf Step 2.1.} If $x_i \in V(G_{t-1})$, then we take $y_i = x_i$, and we set $x_j = x_j$ for $j > i$.

\smallskip

{\bf Step 2.2.} If $x_i \in V(G_{t}) \setminus V(G_{t-1})$ is the clone of $x'_i \in V(G_{t-1})$, then we have two possibilities; either $x'_i \neq y_{i-1}$, or $x'_i = y_{i-1}$.

\smallskip

{\bf Step 2.2.1.} If $x'_i \neq y_{i-1}$ and $x'_i \neq x_{i+1}$, then using Corollary \ref{cone}, we set $(y_1, \ldots, y_{i-1}, y_i = x'_i, x_{i+1} = x_{i+1}, \ldots, x_k = x_k)$ as a new burning
sequence for $G_t$.

\smallskip

If $x'_i \neq y_{i-1}$ and $x'_i = x_{i+1}$, then we set $(y_1, \ldots, y_{i-1}, y_i = x'_i , x_{i+1} = x, x_{i+2} = x_{i+2}, \ldots, x_k = x_k)$ as a new burning sequence for $G_t$, in which $x$ is
an unburned node at stage $i$ in $G_{t-1}$ such that $x\neq x_j$, for any $j> i$, if such a node is available. We can do this since $N[x_i] = N[x'_i] = N[x_{i+1}]$, and burning two successive fire
sources that have exactly the same closed neighbourhoods does not change the course of fire. In fact, it may even slow down the burning process.

If such a node $x$ is not available, then we set $(y_1, \ldots, y_{i-1}, y_i = x'_i , x_{i+1} = x_{i+2},  \ldots, x_{k-1} = x_k,  x_k =z)$ as a new burning sequence for $G_t$, in which $z$ is an
unburned node at step $k-1$. We know that such a node does exist, since by assumption, $b(G_t) = k$.

\smallskip

{\bf Step 2.2.2.} If $x'_i = y_{i-1}$, then we set $(y_1, \ldots, y_{i-1}, y_i = x, x_{i+1} = x_{i+1}, \ldots, x_k = x_k)$ as a burning sequence for $G_t$, in which $x$ is an unburned node at stage
$i-1$ in $G_{t-1}$ such that $x\neq x_j$, for any $j> i$, if such a node is available. We can do this since $N[x_i] = N[x'_i] = N[y_{i-1}]$, and burning two successive fire sources that have exactly
the same closed neighbourhoods does not change the course of fire.

If such a node $x$ is not available, then we set $(y_1, \ldots, y_{i-1}, x_i = x_{i+1}, \ldots, x_{k-1} = x_k,  x_k = z)$ as a burning sequence for $G_t$, in which $z$ is an unburned node at step
$k-1$. We know that such a node does exist, since by assumption, $b(G_t) = k$. Go to Step 1. Note that, in such a case, we have not yet defined $y_i$, and we only replaced one of the possible options
for $y_i$, that is $x_i$, by the next node in the sequence.

\smallskip

{\bf Step 3.} We have two possibilities for $x_k$; either $x_k \in V(G_{t-1})$, or $x_k \not\in V(G_{t-1})$.

\smallskip

{\bf Step 3.1.} If $x_k \in V(G_{t-1})$, then we set $y_k = x_k$. Return $(y_1, y_2,\ldots, y_k)$.

\smallskip

{\bf Step 3.2.} If $x_k$ is the clone of $x'_k \in V(G_{t})$, then we have two possibilities; either $x'_k = y_{k-1}$, or $x'_k \neq y_{k-1}$.

\smallskip

{\bf Step 3.2.1.} If $x'_k \neq y_{k-1}$, then by Corollary \ref{cone}, we set $y_k = x'_k$. Return $(y_1, y_2,\ldots, y_k)$.

\smallskip

{\bf Step 3.2.2.} If $x'_k = y_{k-1}$, then we set $y_k = x$, where $x$ is an unburned node in $G_{t-1}$ at stage $k-1$, if such a node is available. Return $(y_1, y_2,\ldots, y_k)$ as a burning
sequence for $G_t$, since the fire spread from $y_{k-1} = x'_k$ will burn $x_k$ in the $k$-th step, and by assumption any other unburned node will burn at step $k$. Hence $x_k$ is interchangeable by
any unburned node from stage $k-1$.

If such a node $x$ is not available, then we set $y_k = x_k$, and return $(y_1, y_2,\ldots, y_k)$.

\medskip

Note that Step 2.2.2 of the above algorithm does not generate an infinite loop, and the algorithm ends after finite number of steps. This follows since for burning $G_{t-1}$ we need at least $k-1$
steps. Thus, for any $i \leq k-1$ we always have an unburned node $x$ in $G_{t-1}$ at stage $i-1$. If $x$ is not in the sequence, then it will be chosen as $y_i$ through the steps defined above. If
such a node $x$ is only found in the sequence, then after finite number of performing Step 2.2.2, we choose $x$ as $y_i$, if it is necessary.

Suppose that for every optimal burning sequence $(x_1, x_2,\ldots, x_k)$ of $G_0$ all the neighbours of $x_k$ are burned in the $k$-th step. We claim that $b(G_1) = b(G_0) + 1$. Assume not; that is,
$b(G_1) = b(G_0)$. Let $(y_1, y_2,\ldots, y_k)$ be an optimal burning sequence for $G_1$ that is obtained from an optimal burning sequence $(z_1, z_2, \ldots, z_k)$ for $G_1$ by the algorithm above.
Hence, $\{y_1, \ldots, y_k\}\subseteq G_0$. Otherwise, it implies that $b(G_0) = k-1$, which is a contradiction. But, then to burn $y'_k \in V(G_1)$ (the clone of $y_k$) by the end of the $k$-th
step, one of the nodes in the neighbourhood of $y_k$ must be burned in an earlier stage, which is a contradiction with the assumption. Therefore, in this case $b(G_1) = b(G_0)$ is impossible, and
hence, $b(G_1) = b(G_0) + 1$.

Conversely, suppose that $b(G_1) = b(G_0) + 1$, and $(x_1, x_2,\ldots, x_k)$ is an optimal burning sequence for $G_0$. If $x_k$ has a neighbour that is burned at stage $k-1$, then $x'_k$ is also burned
at stage $k$. Therefore, $(x_1, x_2,\ldots, x_k)$ is also a burning sequence for $G_1$, and we have that $b(G_1) = b(G_0)$, which is a contradiction. Thus, $b(G_1) = b(G_0) + 1$, if and only if for every
optimal burning sequence of $G_0$, say $(x_1, x_2,\ldots, x_k)$, all the neighbours of $x_k$ are burned in stage $k$. By induction, we can conclude that $b(G_t) = b(G_0) + 1$ if and only if for every
optimal burning sequence of $G_0$, say $(x_1, x_2,\ldots, x_k)$, all the neighbours of $x_k$ are burned in stage $k$. Since starting from any graph $G_0$, for any $t\geq 1$, either $b(G_t) = b(G_0)$, or $b(G_t)
= b(G_0) + 1$, we conclude that $b(G_t) = b(G_0)$ if and only if for every optimal burning sequence of $G_0$, say $(x_1, x_2,\ldots, x_k)$ one of the neighbours of $x_k$ is burned at stage $k-1$.
\end{proof}

We finish this section with an example that illustrates Theorem \ref{ilttt}. Let $P_n$ be a path on $n$ nodes such that $\lceil {n}^{1/2}\rceil = k$, for some positive integer $k$. Then by Theorem
\ref{path}, we know that $b(P_n) = k$. Moreover, if $(x_1, x_2,\ldots, x_k)$ is an optimal burning sequence for $P_n$, then burning $P_n$ is equivalent to decomposing $P_n$ into paths of orders at
most $1, 3, \ldots, 2k-1$, in which each path is a rooted path of radius at most $k-i$ and with root $x_i$, for some $1\leq i\leq k$. Thus, we can easily see that $x_k$ is the path of order $1$ in
such a decomposition for $P_n$ in terms of neighbourhoods of $x_i$'s. There are two possibilities for $n$; either $n=k^2$, or $n \neq k^2$.

If $n= k^2$, then it implies that the order of each path in decomposing $P_n$ is exactly equal to $2(k-i)+1$, for some $1\leq i\leq k$. Therefore, the end points of such paths are burned in the
$k$-th steps. Hence, both neighbours or the only neighbour of $x_k$ must burn in the $k$-th step, depending on the position of $x_k$ in $P_n$. Thus, by Theorem \ref{ilttt}, if $G_0 = P_n$ in the ILT
model, then we have that $b(G_t) = b(P_n) + 1 = k+1$, for $t\geq 1$.

On the other hand, if $n\neq k^2$, then, there is at least one $i$ for which the rooted path with root $x_i$ is of order less than $2(k-i) + 1$. That is, one of the end points of this path called $x$
is not burned at the $k$-th step. If in decomposing $P_n$, we choose $x_k$ to be the neighbour of $x$, then we have a burning sequence for $P_n$ such that at least one of the neighbours of $x_k$ is
not burned at step $k$. Therefore, by Theorem \ref{ilttt}, if $G_0 = P_n$ in the ILT model, then we have that $b(G_t) = b(P_n)= k$.

\section{Conclusions and future work}

We introduced a new graph parameter, the burning number of a graph, written $b(G).$ The burning number measures how rapidly social contagion spreads in a given graph. We gave a characterization of
the burning number in terms of decompositions into trees, and gave bounds on the burning number which allow us to compute it for a variety of graphs. We showed the strong connection between the
burning number and the distance domination, that we use it for finding bounds on the burning number, as well as proving the Nordhaus-Gaddum type-results on the burning number of a graph and its
complement. We determined the burning number in the Iterated Local Transitive model for social networks.

Several problems remain on the burning number. We conjecture that for a connected graph $G$ of order $n$, $b(G)\leq \lceil {n}^{1/2}\rceil$. Determining the burning number remains open for many
classes of graphs, including trees and disconnected graphs. It remains open to consider the burning number in real-world social networks such as Facebook or LinkedIn. As Theorem~\ref{ilttt} suggests,
the burning number of on-line social networks is likely of constant order as the network grows over time. We remark that burning number generalizes naturally to directed graphs; one interesting
direction is to determine the burning number on Kleinberg's small world model \cite{klein}, which adds random directed edges to the Cartesian grid.

A simple variation which leads to complex dynamics is to change the rules for nodes to burn. As in graph bootstrap percolation \cite{bbm}, the rules could be varied so nodes burn only if they are
adjacent to at least $r$ burned neighbors, where $r>1.$ We plan on studying this variation in future work.

\end{document}